\documentclass{amsart}
\usepackage{amssymb}
\usepackage{lipsum}

\theoremstyle{plain}
\newtheorem{theorem}{Theorem}[section]
\newtheorem{proposition}{Proposition}[section]
\newtheorem{lemma}{Lemma}[section]
\newtheorem{corollary}{Corollary}[section]
\theoremstyle{definition}
\newtheorem{example}{Example}[section]
\newtheorem{definition}{Definition}[section]
\theoremstyle{remark}

\numberwithin{equation}{section}

\usepackage{tikz,xcolor,hyperref}

\definecolor{lime}{HTML}{A6CE39}
\DeclareRobustCommand{\orcidicon}{%
	\begin{tikzpicture}
		\draw[lime, fill=lime] (0,0)
		circle [radius=0.16]
		node[white] {{\fontfamily{qag}\selectfont \tiny ID}};
		\draw[white, fill=white] (-0.0625,0.095)
		circle [radius=0.007];
	\end{tikzpicture}
	\hspace{-2mm}
}

\foreach \x in {A, ..., Z}{%
	\expandafter\xdef\csname orcid\x\endcsname{\noexpand\href{https://orcid.org/\csname orcidauthor\x\endcsname}{\noexpand\orcidicon}}
}

\begin{document}
	
	\title{Commutativity of factor ring $R/P$ via $\ast$-reverse $P$-derivations}
	
	

		\author[D. Kumar]{Deepak Kumar}
	
	\address{ 
		Department of Mathematics \\ 
		Punjabi University \\ 
		Patiala-147002\\
		India}
\email{deep\_math1@yahoo.com}

		\author[S. Singh]{Sukhchain Singh}
	\address{ 
		Department of Mathematics \\ 
		Punjabi University \\ 
		Patiala-147002\\
		India\\
		ORCID: 0000-0001-6127-0731}
	\email{sukhchain\_rs@pbi.ac.in}


\maketitle

\begin{abstract}
	  
	  In this paper, we establish commutativity criteria for the quotient ring $R/P$ through pairs of $\ast$-reverse $P$-derivations satisfying suitable commutator and anti-commutator differential identities. As applications, a number of results concerning $\ast$-reverse derivations on prime rings are recovered.  
\end{abstract}

\section{Introduction}

Throughout this article, $R$ denotes a ring with center $Z(R)$. An ideal $P$ of $R$ is called a \emph{prime ideal} if $aRb\subseteq P$ implies that either $a\in P$ or $b\in P$. A ring $R$ is said to be a \emph{prime ring} whenever its zero ideal $(0)$ is prime. An \emph{involution} on $R$ is an anti-automorphism $\ast$ satisfying $(a^\ast)^\ast=a$ for every $a\in R$. An element $a\in R$ is called \emph{symmetric} (respectively, \emph{skew-symmetric}) if $a^\ast=a$ (respectively, $a^\ast=-a$). The sets of all symmetric and skew-symmetric elements of $R$ are denoted by $H(R)$ and $S(R)$, respectively.

For any nonempty subset $D$ of $R$, an involution $\ast$ is said to be of $D$-\emph{second kind} if $S(R)\cap Z(R)\nsubseteq D$; otherwise, it is called of $D$-\emph{first kind} (see \cite{Rehman24}). In the special case $D={0}$, these notions reduce to the usual involutions of the second and first kind, respectively. Throughout the paper, $[a,b]$ and $a\circ b$ denote the commutator $ab-ba$ and the anti-commutator $ab+ba$, respectively. For any $a,b,c\in R$, we have following identities which works as an important tools in our proofs: 	
\begin{itemize}
	\item[(i)]  $(a\circ bc)=(a\circ b)c-b[a,c]=b(a\circ c)+[a,b]c$
	\item[(ii)]  $(ab\circ c)=a(b\circ c)-[a,c]b=(a\circ c)b+a[b,c]$
	\item[(iii)] $[a,bc]=[a,b]c+b[a,c] $.
	\end{itemize}\par 

 An additive mapping $d:R\to R$ is called a \emph{derivation} if $d(ab)=d(a)b+ad(b)$ for all $a,b\in R$. In an attempt to develop an analogous theory, Herstein \cite{Herstein} introduced the notion of a \emph{reverse derivation} as an additive mapping $d:R\to R$ satisfying $d(ab)=d(b)a+bd(a)$ for all $a,b\in R$. He proved that a noncommutative prime ring admits no nonzero reverse derivation. Later, Samman \cite{Samman2006} showed that every reverse derivation on a semiprime ring is a central derivation. These results reveal that the behaviour of reverse derivations on prime and semiprime rings is rather restrictive.

The presence of an involution naturally leads to the study of differential mappings compatible with this additional structure. Bre\v{s}ar and Vukman \cite{bresar1989} introduced the notion of a $\ast$-derivation on rings with involution. Motivated by their work, Bhushan et al. \cite{bhushan2020} proposed the concept of a $\ast$-reverse derivation. An additive mapping $\mathfrak{d}:R\to R$ is called a $\ast$-reverse derivation if $\mathfrak{d}(ab)=\mathfrak{d}(b)a^\ast+b^\ast\mathfrak{d}(a)$ for all $a,b\in R$. An example of such a mapping is given by $a\mapsto[a^\ast,r]$, where $r$ is a fixed element of $R$. This mapping is a $\ast$-reverse derivation but not a derivation, illustrating that the theory of $\ast$-reverse derivations remains nontrivial on prime rings and provides a useful framework for investigating their structural properties.

 On the other hand, Sandhu et al. \cite{Sandhu23} introduced the notion of a $P$-derivation using a prime ideal $P$ of $R$. A mapping $d:R\to R$ is called a $P$-derivation if $d(a+b)-d(a)-d(b)\in P$ and $d(ab)-d(a)b-ad(b)\in P$
for all $a,b\in R$. The concept of a $P$-derivation extends the notion of derivation modulo a prime ideal, while a $\ast$-reverse derivation incorporates the involution structure of the underlying ring. It is therefore natural to investigate mappings that combine these two features, leading to the notion of a $\ast$-reverse $P$-derivation.\par 

The study of differential identities has attracted considerable attention since the pioneering work of Herstein \cite{herstein78}, who showed that a prime ring $R$ with $\mathrm{char}(R)\neq 2$ must be commutative whenever there exists a nonzero derivation $d$ satisfying $[d(x),d(y)]=0$ for all $x,y\in R$. Since then, differential identities associated with various derivations have been widely used to investigate the structure of prime rings and rings possessing prime ideals. Notable contributions in this direction include \cite{Mamouni2018, Bhushan2021, Deepak24, ElMir2020, M.Ashraf2002, M.A.Idrissi2022}. More recently, Khan et al. \cite{Khan2021} extended Herstein's philosophy to involutive rings with prime ideals. One of their result is formulated as follows.

\textit{
	Suppose that $R$ is a ring with an involution $\ast$ of the second kind and $P$ is a prime ideal of $R$ satisfying $S(R)\cap Z(R)\nsubseteq P$ and $\mathrm{char}(R/P)\neq 2$. Let $d_1$ and $d_2$ be derivations of $R$ such that
	$$
	[d_1(x),d_2(x^\ast)]\in P
	$$
	for every $x\in R$. Then at least one of the following alternatives occurs:
	\begin{itemize}
		\item[(a)] $d_1(R)\subseteq P$;
		\item[(b)] $d_2(R)\subseteq P$;
		\item[(c)] the factor ring $R/P$ is a commutative integral domain.
	\end{itemize}
}

 Motivated by these developments, we introduce the notion of a $\ast$-reverse $P$-derivation and study pairs of such mappings through suitable commutator and anti-commutator identities. More precisely, our main results are established under the assumptions that, for all $a\in R$,
\begin{itemize}
	\item[(i)] $[\delta_1(a),\delta_2(a^\ast)]\pm[a,a^\ast]\in P$,
	\item[(ii)] $\delta_1(a)\circ\delta_2(a^\ast)\pm a\circ a^\ast\in P$.
\end{itemize}
These conditions enable us to obtain commutativity criteria for the factor ring $R/P$ and to determine situations in which the images of the associated $\ast$-reverse $P$-derivations are necessarily contained in the prime ideal $P$.

\section{Preliminaries}

In this section, we recall certain definitions and auxiliary results that will be used throughout the paper. We begin by introducing the notion of a $\ast$-reverse $P$-derivation and present an illustrative example highlighting its distinction from a $\ast$-reverse derivation.

\begin{definition}
	A map $\delta: R\to R$ is called $\ast$-reverse $P$-derivation on ring $R$ with involution $\ast$ if for all $a,b\in R$, we have
	\begin{itemize}
		\item[(a)] $\delta(a+b)-\delta(a)-\delta(b)\in P$
		\item[(b)] $\delta(ab)-\delta(b)a^{\ast}-b^{\ast}\delta(a)\in P$.
	\end{itemize}

\begin{example}
	Let a ring $R=Q\times Q$, where $Q$ is any prime ring with an involution-$\ast$. Clearly the ideal $P=\{0\}\times Q$ is a prime ideal of $R$. For any fixed elements $p,q \in Q, q\neq 0$, define a map $\delta:R\to R$  as 
	$$ \delta ((a,b))= ([a^{\ast},p],q) \text{~for~all~}(a,b)\in R$$
	Thus the map $\delta$ is a $\ast$-reverse $P$-derivation on ring $R$ which is not a $\ast$-reverse derivation.
\end{example}
\end{definition}

\begin{lemma}\label{lemmaA}
	
	Let $R$ be a prime ring with involution-$\ast$. If $\delta$ is a $\ast$-reverse P-derivation on $R$ such that $[\delta(a), a^{\ast}]\in P$ for all $a\in R$, then either $\delta(R)\subseteq P$ or $R/P$ is a commutative integral domain.
\end{lemma}
\begin{proof}
	Given that 
	\begin{equation}\label{L11}
		[\delta(a), a^{\ast}]\in P \text{~for~all~} a\in R
	\end{equation}
	Linearizing the above relation, we find
	\begin{equation}\label{L1a}
		[\delta(a), b^{\ast}]+[\delta(b), a^{\ast}] \in P\text{~for~all~}a,b\in R
	\end{equation}
	Changing $b$ into $ba$ for any $a\in R$, in equation (\ref{L1a}), we get
	\begin{equation}\label{L1b}
		[\delta(a), a^{\ast}]b^{\ast}+ a^{\ast}[\delta(a), b^{\ast}]+[\delta(a)b^{\ast}+a^{\ast}\delta(b), a^{\ast}] \in P\text{~for~all~}a,b\in R
	\end{equation}
	On solving and using the equation (\ref{L11}), to find
	\begin{equation}\label{L1c}
		a^{\ast}[\delta(a), b^{\ast}]+\delta(a)[b^{\ast}, a^{\ast}] + a^{\ast}[\delta(b),a^{\ast}]\in P\text{~for~all~}a,b\in R
	\end{equation}
	Using the equation (\ref{L1a}), we demonstrate
	\begin{equation}\label{L1d}
		\delta(a)[b^{\ast}, a^{\ast}] \in P\text{~for~all~}a,b\in R
	\end{equation}
	Changing $b$ into $b^{\ast} r^{\ast}$ for any $r\in R$, in equation (\ref{L1d}), we get
	\begin{equation}\label{L1e}
		\delta(a)r[b, a^{\ast}] \in P\text{~for~all~}a,b,r\in R
	\end{equation}

Now let $A=\{a~| ~\delta(a)\in P\}$ and $B=\{a~|~ [b,a^{\ast}]\in P~ \forall b\in R\}$. Clearly $A$ and $B$ both are additive subgroups of ring $R$ and $R=A\cup B$. But $R$ cannot be union of two proper subgroups. Thus either $A=R$ or $B=R$, means either $\delta(R)\subseteq P$ or $[b,a^{\ast}]\in P$ for all $a,b\in R$ i.e., $\overline{[b,a]}= \overline{0}$ in ring $R/P$ for all $a,b\in R$, which means $R/P$ is commutative.

\end{proof}

\begin{lemma}\label{lemmaB}
	Let $R$ be an arbitrary ring with prime ideal $P.$ If $\delta:R\to R$ is a $\ast$-reverse $P$-derivation of $R,$ then for all  $z\in Z(R)$,  $\overline{\delta(z)}\in  Z(R/P).$
\end{lemma}
\begin{proof}
	Let $z\in Z(R)$. Then we have
	\begin{equation}\label{A}
		\delta(az)-\delta(z)a^{\ast}-z^{\ast}\delta(a)\in P\text{~for~all~}a\in R.
	\end{equation}
	and
	\begin{equation}\label{B}
		\delta(za)-\delta(a)z^{\ast}-a^{\ast}\delta(z)\in P\text{~for~all~}a\in R.
	\end{equation}
	Using the fact that $z^{\ast} \in Z(R)$ if $z\in Z(R)$ and comparing (\ref{A}) and (\ref{B}), we get $[\delta(z),a^{\ast}]\in P$ for all $a\in R.$ It forces that $\overline{\delta (z)}\in  Z(R/P).$
\end{proof}

\section{Main Results}

This section is devoted to the study of $\ast$-reverse $P$-derivations satisfying several differential identities involving commutators and anti-commutators. Under suitable assumptions on the involution and the prime ideal, we obtain conditions that force the quotient ring $R/P$ to be commutative.

\begin{proposition}\label{thm1}
	Let $R$ be a ring with involution $\ast$, $P$ is a prime ideal of $R$ such that char($R/P$)$\neq 2.$ If $\delta_1,\delta_2$ are two $\ast$-reverse $P$-derivations on $R$ such that $[\delta_1(a),\delta_2(b)]\pm [a^\ast, b^\ast]\in P$ for all $a,b\in R$ then either
	\begin{itemize}
		\item[(i)] $\delta_1 (R)\subseteq P$, $\delta_2 (R)\subseteq P$ or
		\item[(ii)] $R/P$ is commutative integral domain.
	\end{itemize}
	  
\end{proposition}
\begin{proof}
	We have 
	\begin{equation}\label{4a}
		[\delta_1(a),\delta_2(b)]\pm[a^\ast ,b^\ast]\in P \text{~for~all~} a,b \in R
	\end{equation}
	Replacing $b$ into $cb$, where $c\in R$ in the equation (\ref{4a}), we get
	\begin{equation}\label{4b}
		\begin{split}
			[\delta_1(a), \delta_2(b)]c^\ast +\delta_2(b)[\delta_1(a),c^\ast]+[\delta_1(a),b^\ast]\delta_2(c) \\ +b^\ast[\delta_1(a), \delta_2(c)] 
			\pm b^\ast [a^\ast,c^\ast]\pm[a^\ast,b^\ast]c^\ast \in P
		\end{split}		 
	\end{equation}
	Using (\ref{4a}) in the equation (\ref{4b}), we find
	\begin{equation}\label{4c}
		\delta_2(b)[\delta_1(a),c^\ast]+[\delta_1(a),b^\ast]\delta_2(c)\in P
	\end{equation}
	Taking $c=\delta_1(a)^\ast$ in (\ref{4c}) we get
	\begin{equation}\label{4d}
		[\delta_1(a),b^\ast]\delta_2(\delta_1(a)^\ast)\in P
	\end{equation}
	For any $r\in R$, replacing $b$ into $br$ in the equation (\ref{4d}) to get
	\begin{equation}\label{4e}
		[\delta_1(a),r^\ast]b^\ast \delta_2(\delta_1(a)^\ast) \in P
	\end{equation}	
	Now changing $b$ into $\delta_1(r)^\ast b$ and using (\ref{4e}), we get
	\begin{equation}\label{4g}
		[\delta_1(a),r^\ast] b^\ast [\delta_2(\delta_1(a)^\ast), \delta_1(r)] \in P
	\end{equation}
	From (\ref{4a}), we find 
	\begin{equation}\label{4g1}
		[\delta_2(\delta_1(a)^\ast), \delta_1(r)]\pm [\delta_1(a), r^{\ast}]\in P
	\end{equation}
	Using (\ref{4g}) and (\ref{4g1}), we obtain
	\begin{equation}\label{4h}
		[\delta_1(a),r^\ast] b^\ast [\delta_1(a),r^\ast] \in P
	\end{equation}
	It is reduced to 
	\begin{equation}\label{4h}
		[\delta_1(a),r^\ast] R [\delta_1(a),r^\ast] \subseteq P
	\end{equation}
	As $P$ is a prime ideal of $R$ thus by taking $r=a$, we obtain
	\begin{equation}\label{4i}
		[\delta_1(a),a^\ast]\in P \text{~for~all~}x,y\in R
	\end{equation}
	Hence by using the Lemma \ref{lemmaA}, either $\delta_1(R)\subseteq P$ or $R/P$ is commutative. Similarly by taking suitable substitutions we obtain either $\delta_2(R)\subseteq P$ or $R/P$ is commutative. 
	
\end{proof}
\begin{corollary}\label{cor1}
		Let $R$ be a ring with involution $\ast$, $P$ is a prime ideal of $R$ such that char($R/P$)$\neq 2$ and $\delta$ be a $\ast$-reverse P-derivation. If $[\delta(a),\delta(b)]\pm[a^\ast ,b^\ast]\in P$ for all $a,b \in R$, then either $\delta (R)\subseteq P$ or $R/P$ is commutative.
\end{corollary}

\begin{corollary}\label{cor1a}
Let $R$ be a prime ring such that char($R$)$\neq 2$ and $d_1$, $d_2$ are $\ast$-reverse derivations. If $[d_1(a),d_2(b)]\pm [a^\ast ,b^\ast]=0$ for all $a,b \in R$, then $R$ is commutative.
\end{corollary}

\begin{theorem}\label{thm2}
	Let $R$ be a ring with involution $\ast$ of $P$-second kind where $P$ is a prime ideal of $R$ such that char($R/P$)$\neq 2.$ If $\delta_1,\delta_2$ are two $\ast$-reverse $P$-derivations on $R$ such that $[\delta_1(a),\delta_2(a^\ast)]\pm [a,a^\ast]\in P$ for all $a\in R$ then either 	\begin{itemize}
		\item[(i)] $\delta_1 (R)\subseteq P$, $\delta_2 (R)\subseteq P$ or
	\item[(ii)] $R/P$ is commutative integral domain.
	\end{itemize}
\end{theorem}

\begin{proof}
	We have given for all $a\in R$
\begin{equation}\label{9a}
	 [\delta_1(a),\delta_2(a^\ast)]\pm [a,a^\ast]\in P
\end{equation}	
Polarizing the above relation to have for all $a,b\in R$
\begin{equation}\label{9b}
	[\delta_1(a),\delta_2(b^\ast)]+[\delta_1(b),\delta_2(a^\ast)]\pm [a,b^\ast]\pm [b,a^\ast]\in P
\end{equation}
Replacing $a$ into $ah$ for any $h\in H(R)\cap Z(R)$ and using the lemma \ref{lemmaB} and the equation (\ref{9b}), we get
\begin{equation}\label{9c}
	\delta_1(h)[a^\ast,\delta_2(b^\ast)]+[\delta_1(b),a]\delta_2(h)\in P
\end{equation}
Now replacing $a$ into $ak$ for any nonzero $k\in S(R)\cap Z(R)$ 
\begin{equation}\label{9df}
		k(-\delta_1(h)[a^\ast,\delta_2(b^\ast)]+[\delta_1(b),a]\delta_2(h))\in P
\end{equation}
Using the condition $S(R)\cap Z(R)\nsubseteq P$ and the primeness of $P$, we capture
\begin{equation}\label{9d}
	-\delta_1(h)[a^\ast,\delta_2(b^\ast)]+[\delta_1(b),a]\delta_2(h)\in P
\end{equation}
Adding the equation (\ref{9c}) and (\ref{9d}), we have
\begin{equation}\label{9e}
	2[\delta_1(b),a]\delta_2(h)\in P
\end{equation}
Since char($R/P$)$\neq 2$, so the above equation reduces to
\begin{equation}\label{9f}
	[\delta_1(b),a]\delta_2(h)\in P
\end{equation}
Taking $a=ar$ for any $r\in R$ and simplifying it by using (\ref{9f})
\begin{equation}\label{9f1}
	[\delta_1(b),a]r\delta_2(h)\in P \text{~i.e,~} [\delta_1(b),a]R\delta_2(h)\subseteq P
\end{equation}
It concludes for all $a,b\in R$ and $h\in H(R)\cap Z(R)$
\begin{equation}\label{9f2}
	\text{~either~} [\delta_1(b),a]\subseteq P \text{~or~} \delta_2(h) \subseteq P
\end{equation}

 For the first case,  we get either $\delta_1(R)\subseteq P$  or $R/P$ is commutative, by using the lemma \ref{lemmaA}.
In second case, we have $\delta_2(h)\in P$, changing $h=k^2$ for any $k\in S(R)\cap Z(R)$ to find $\delta_2(k)\in P$ for all $k\in S(R)\cap Z(R)$. Similarly we can easily obtain  $\delta_1(k)\in P$ for all $k\in S(R)\cap Z(R)$.\\
Now replacing $a$ into $ak$ in the equation (\ref{9b}) we have
\begin{equation}\label{9g}
	k(-[\delta_1(a),\delta_2(b^\ast)]+[\delta_1(b),\delta_2(a^\ast)])\pm k([a,b^\ast]-[b,a^\ast]) \in P
\end{equation}
Eliminating $k$ in the expression by using primesness of $P$ and the assumption $S(R)\cap Z(R)\nsubseteq P$, we have 
\begin{equation}\label{9h}
	-[\delta_1(a),\delta_2(b^\ast)]+[\delta_1(b),\delta_2(a^\ast)]\pm [a,b^\ast]\mp [b,a^\ast]\in P
\end{equation}
Using both the equations (\ref{9b}) and (\ref{9h}), we obtain
\begin{equation}\label{9i}
	[\delta_1(b),\delta_2(a^\ast)]\pm [a,b^\ast]\in P
\end{equation}
After taking $a^\ast$ instead of $a$ and interchanging variables, the above relation can be written as
\begin{equation}\label{9j}
	[\delta_1(a),\delta_2(b)]\mp [a^\ast ,b^\ast]\in  P
\end{equation}
Now by following the Proposition \ref{thm1}, we conclude that $R/P$ is commutative integral domain.
\end{proof}

\begin{corollary}
		Let $R$ be a ring with involution $\ast$ of $P$-second kind where $P$ is a prime ideal of $R$ such that char($R/P$)$\neq 2.$ If $\delta$ is $\ast$-reverse $P$-derivation on $R$ such that $[\delta(a),\delta(a^\ast)]\pm [a,a^\ast]\in P$ for all $a\in R$ then either 	$\delta (R)\subseteq P$ or
		$R/P$ is commutative integral domain.
\end{corollary}
\begin{corollary}
		Let $R$ be a ring with involution $\ast$ of $P$-second kind where $P$ is a prime ideal of $R$ such that char($R/P$)$\neq 2.$ If $\delta$ is $\ast$-reverse $P$-derivation on $R$ such that $[\delta(a),\delta(a^\ast)]\pm [a,a^\ast]\pm a\circ a^\ast \in P$ for all $a\in R$ then either 	$\delta (R)\subseteq P$ or
	$R/P$ is commutative integral domain.
	
\end{corollary}
\begin{proof}
	We have for all $a\in R$
	 $$[\delta(a),\delta(a^\ast)]\pm [a,a^\ast]\pm a\circ a^\ast\in P$$
	 Interchanging $a$ and $a^\ast$ in the above relation, we find
	 $$[\delta(a),\delta(a^\ast)]\pm [a,a^\ast]\mp a\circ a^\ast\in P$$
	 From last two equations we conclude
	 $$[\delta(a),\delta(a^\ast)]\pm [a,a^\ast]\in P$$
	 Therefore by using the Theorem \ref{thm2} we get the required result.
\end{proof}

\begin{proposition}\label{thm3}
	Let $R$ be a ring with involution $\ast$, $P$ is a prime ideal of $R$ such that char($R/P$)$\neq 2$ and $Z(R/P)\neq \{0\}$. If $\delta_1,\delta_2$ are two $\ast$-reverse $P$-derivations on $R$ such that $\delta_1 (a)\circ \delta_2 (b)\pm a^{\ast}\circ b^{\ast} \in P$  for all $a,b\in R$ then either
		\begin{itemize}
			\item[(i)] $\delta_1 (R)\subseteq P$, $\delta_2 (R)\subseteq P$ or
		\item[(ii)] $R/P$ is commutative integral domain.
	\end{itemize}

\end{proposition}

\begin{proof}
	We have given that
	\begin{equation}\label{7a}
		\delta_1(a)\circ \delta_2(b)\pm a^{\ast}\circ b^{\ast}\in P \text{~for~all~} a,b \in R
	\end{equation}
	Taking  $ab$ in place of $a$ in the above equation
	\begin{equation}\label{7b}
		\delta_1 (b)a^\ast \circ \delta_2(b)+b^\ast \delta_1(a)\circ \delta_2 (b)\pm b^\ast a^\ast \circ b^\ast \in P
	\end{equation}
	On expanding anti-commutators (\ref{7b}), we get
	
	\begin{equation}\label{7c}
		\begin{split}
			(\delta_1(b)\circ \delta_2(b))a^\ast +\delta_1 (b)[a^\ast, \delta_2 (b)]+b^\ast (\delta_1(a)\circ \delta_2 (b))\\
			-[b^\ast, \delta_2 (b)]\delta_1 (a)\pm b^\ast (a^\ast \circ b^\ast)\in P
		\end{split}	
	\end{equation}
	Using (\ref{7a}) and the previous equation (\ref{7c}), we find
	\begin{equation}\label{7c2}
			(\delta_1(b)\circ \delta_2(b))a^\ast +\delta_1 (b)[a^\ast, \delta_2 (b)]
			-[b^\ast, \delta_2 (b)]\delta_1 (a)\in P
	\end{equation}
	It can be also be written as
	\begin{equation}\label{7d}
		(\delta_1(b)\circ \delta_2(b))a^\ast -\delta_1 (b)[\delta_2 (b),a^\ast]+[\delta_2 (b),b^\ast]\delta_1 (a)\in P
	\end{equation}
	For any $r\in R$, replacing $a$ with $ra$ in (\ref{7d}), we construct
	\begin{equation}\label{7e}
		-\delta_1 (b)a^\ast [\delta_2 (b),r^\ast]+[\delta_2 (b),b^\ast]a^\ast \delta_1 (r) \in P
	\end{equation}
By putting $r=z\in Z(R)$ and replacing $a$ into $a^\ast$ in the equation (\ref{7e}), we obtain for all $z\in Z(R), a,b\in R$
\begin{equation}\label{7na}
	[\delta_2 (b),b^\ast]a \delta_1 (z)\in P \text{~i.e,~} 	[\delta_2 (b),b^\ast]R \delta_1 (z)\subseteq P
\end{equation}
Since $P$ is prime in $R$, so for all $z\in Z(R), b\in R$, we have either
\begin{equation}\label{7nb}
		[\delta_2 (b),b^\ast]\in P\text{~or~} \delta_1 (z)\in P
\end{equation}
   If $\delta_1 (z)\in P$ for all $z\in Z$, then by taking $b=z$ for any $z\in Z(R)$ in the hypothesis we obtain $2za\in P$ for all $a\in R$ i.e, $\overline{2za}=\overline{0}$ in $R/P$. Since $R/P$ is prime ring and char($R/P$)$\neq 2$  which implies that $\bar{z}=\bar{0}$ in $R/P$, that leads to contradiction to our assumption $Z(R/P)\neq \{0\}$.\par
  On the other hand if $[\delta_2 (b),b^\ast]\in P$ for all $b\in R$, then we get the desired outcome by using lemma \ref{lemmaA}.
\end{proof}

\begin{corollary}
	Let $R$ be a ring with involution $\ast$, $P$ is a prime ideal of $R$ such that char($R/P$)$\neq 2$ and $Z(R/P)\neq \{0\}$. Let $\delta$ be a $\ast$-reverse P-derivation. If $\delta (a)\circ \delta (b)\pm a^{\ast}\circ b^{\ast} \in P$  for all $a,b \in R$, then either $\delta (R)\subseteq P$ or $R/P$ is commutative.
\end{corollary}
\begin{corollary}
		Let $R$ be a prime ring with involution $\ast$, $d_1, d_2$ be two $\ast$-reverse derivations on $R$ and char($R$)$\neq 2$ . If $d_1 (a)\circ d_2 (b)\pm a^{\ast}\circ b^{\ast}=0$ and $Z(R)\neq \{0\}$, then $R$ is commutative.
\end{corollary}

\begin{theorem}\label{thm4}

		Let $R$ be a ring with involution $\ast$ of $P$-second kind where $P$ is a prime ideal of $R$ such that char($R/P$)$\neq 2.$ If $\delta_1,\delta_2$ are two $\ast$-reverse $P$-derivations on $R$ such that $\delta_1 (a)\circ \delta_2 (a^{\ast})\pm a\circ a^{\ast} \in P$ for all $a\in R$ then either
			\begin{itemize}
			\item[(i)] $\delta_1 (R)\subseteq P$, $\delta_2 (R)\subseteq P$ or
		\item[(ii)] $R/P$ is commutative integral domain.
		\end{itemize}
		
\end{theorem}
\begin{proof}
	We have 
	\begin{equation}\label{8a}
		\delta_1 (a)\circ \delta_2 (a^{\ast})\pm a\circ a^{\ast}\in P \text{~for~all~} a\in R
	\end{equation}
Taking $a+b$ instead of $a$ for any $b\in R$ in the equation (\ref{8a}), we get
\begin{equation}\label{8b}
	\delta_1(a)\circ \delta_2(b^\ast)+\delta_1(b)\circ \delta_2(a^\ast)\pm a\circ b^\ast \pm b\circ a^\ast \in P
\end{equation}
Changing $a$ into $ah$ where $h\in H(R)\cap Z(R)$ in the equation (\ref{8b}) and using the Lemma \ref{lemmaB} we obtain

\begin{equation}\label{8b2}
	\begin{split}
			(\delta_1(a)\circ \delta_2(b^\ast))h+\delta_1(h)(a^\ast \circ \delta_2(b^\ast))+(\delta_1(b)\circ \delta_2(a^\ast))h\\
			+(\delta_1(b)\circ a)\delta_2(h)\pm (a\circ b^\ast)h \pm (b\circ a^\ast)h \in P
	\end{split}
\end{equation}
Multiplying $h$ by (\ref{8b}) and using it in (\ref{8b2}), gives us
\begin{equation}\label{8c}
	\delta_1(h)(a^\ast \circ \delta_2(b^\ast))+(\delta_1(b)\circ a)\delta_2(h)\in P
\end{equation}
For any nonzero $k\in S(R)\cap Z(R)$ changing $a$ into $ak$ in the equation (\ref{8c}) and using primeness of $P$, we get
\begin{equation}\label{8d}
	-\delta_1(h)(a^\ast \circ \delta_2(b^\ast))+(\delta_1(b)\circ a)\delta_2(h)\in P
\end{equation}
Adding the equation (\ref{8c}) and the equation (\ref{8d}) and using char($R/P$)$\neq 2$, we find 
\begin{equation}\label{8e}
	(\delta_1(b)\circ a)\delta_2(h)\in P
\end{equation}
Replacing $a$ with $ar$ for any $r\in R$  in (\ref{8e}) and using an anti-commutator property, we get
\begin{equation}\label{8e1}
	((\delta_1(b)\circ a)r-a[\delta_1(b),r])\delta_2(h)\in P
\end{equation}
Using the lemma \ref{lemmaB} and the equation (\ref{8e1}), we find for all $a,b,r \in R$ and $h\in H(R)\cap Z(R)$,
\begin{equation}\label{8e2}
	a[\delta_1(b),r]\delta_2(h)\in P
\end{equation}
Since $P$ is prime ideal, so we have for all $b,r \in R$ and $h\in H(R)\cap Z(R)$,
\begin{equation}\label{8e3}
	[\delta_1(b),r]\delta_2(h)\in P \text{~i.e.,~} [\delta_1(b),r]R\delta_2(h)\subseteq P
\end{equation}
It implies that either $[\delta_1(b),r]\in P$ for all $b,r\in R$ or $\delta_2(h)\in P$ for all $h\in H$. If $[\delta_1(b),r]\in P$ for all $b,r\in R$, then using lemma \ref{lemmaA} we find the desired result. \\
If we have $\delta_2(h)\in P$, taking $h=k^2$ for $k\in S(R)\cap Z(R)$ to obtain $\delta_2(k)\in P$ for all $k\in S(R)\cap Z(R)$. Similarly we can have  $\delta_1(k)\in P$ for all $k\in S(R)\cap Z(R)$.\\
Now replacing $a$ into $ak$ in the equation (\ref{8b}) we have 
\begin{equation}\label{8f}
	k(-\delta_1(a)\circ \delta_2(b^\ast)+\delta_1(b)\circ \delta_2(a^\ast))\pm k(a\circ b^\ast -b\circ a^\ast)\in P
\end{equation}
Since $P$ is prime and the involution $\ast$ is of $P$-second kind, it provides
\begin{equation}\label{8g}
	-\delta_1(a)\circ \delta_2(b^\ast)+\delta_1(b)\circ \delta_2(a^\ast)\pm (a\circ b^\ast -b\circ a^\ast)\in P
\end{equation}
Comparing (\ref{8b}) and (\ref{8g}), to find
\begin{equation}\label{8h}
\delta_1(b)\circ \delta_2(a^\ast)\pm a\circ b^\ast \in P
\end{equation}
By changing $a$ with $a^\ast$ and interchanging the variables, the equation (\ref{8h}) can be expressed as 
\begin{equation}\label{8i}
	\delta_1(a)\circ \delta_2(b)\pm a^\ast \circ b^\ast \in P \text{~for~all~} a,b\in R
\end{equation}
Hence by using the Proposition \ref{thm3}, ring $R/P$ is commutative.
\end{proof}

\begin{corollary}
Let $R$ be a ring with involution $\ast$ of $P$-second kind where $P$ is a prime ideal of $R$ such that char($R/P$)$\neq 2.$ If $\delta$ is a $\ast$-reverse $P$-derivation on $R$ such that $\delta (a)\circ \delta (a)\pm a\circ a^{\ast} \in P$ for all $a\in R$ then either $\delta (R)\subseteq P$ or $R/P$ is commutative.	
\end{corollary}
\begin{corollary}
	Let $R$ be a prime ring with involution of second kind, $d_1 ,d_2$ are non-zero $\ast$-reverse derivations on $R$ and char($R$)$\neq 2$. If $d_1 (a)\circ d_2 (a)\pm a\circ a^{\ast}=0$ for every $a\in R$, then $R$ is commutative.
\end{corollary}


\begin{thebibliography}{100}
	
	
	\bibitem{Mamouni2018}
	{\sc A.~Mamouni, B.~Nejjar, L.~Oukhtite},
	{\sl Differential identities on prime rings with involution}.
	J. Algebr. Appl. {\bf 17}(9) (2018), 1850163(1-11).
		
	
	\bibitem{bhushan2020}
	{\sc B. Bhushan, G.S. Sandhu, D. Kumar}, {\sl A note on $\ast$-reverse derivation in rings}, Advances in Mathematics: Scientific Journal, \textbf{9 (7)} (2020), 1-6.
	
		\bibitem{Bhushan2021}
	{\sc B.~Bhushan, G.~S.~Sandhu, S.~Ali, D.~Kumar},
	{\sl A classification of generalized derivations in rings with involution.}
	Filomat {\bf 35}(5) (2021), 1439-1452.
	
	
	
		\bibitem{Deepak24}
	{\sc D. Kumar, S. Singh, and B. Bhushan},
	{\sl On *-reverse derivations of involutive rings.}
	 Proceedings of The ICRTMPCS International Conference 2023. Walter de Gruyter GmbH and Co KG. (2024), (p389).
	
		\bibitem{posner1957}
	{\sc  E.C. Posner}
	{\sl Derivations in prime rings}
	Proceedings of the American Mathematical Society 8, 6 (1957),
1093--1100.
	
	\bibitem{Sandhu23}
	{\sc G.S. Sandhu,  A. Boua, N. Ur Rehman},
	{\sl Some results involving P-derivations and prime ideals in rings.}
	ANNALI DELL'UNIVERSITA'DI FERRARA, 69(2). (2023), 587-604.
	
	
	\bibitem{ElMir2020}
	{\sc H. El Mir, A. Mamouni, L. Oukhtite},
	{\sl Commutativity with algebraic identities involving prime ideals},
Commun. Korean Math. Soc., \textbf{35}(3). (2020), 723-731.
	
	
	\bibitem{herstein78}
	{\sc I.N. Herstein}
	\newblock A note on derivation.
	\newblock {\em Canadian Matematics Bulltein 21}, 3 (1978), 369--370.
	
		\bibitem{Herstein}
	{\sc I.N. Herstein}, {\sl Jordan derivations of prime rings}, Proceedings of the American Mathematical Society, \textbf{8} (1957), 1104--1110.
	
	
		\bibitem{M.A.Idrissi2022}
	{\sc 	M. A. Idrissi, L. Oukhtite},
	{\sl Structure of a quotient ring R/P with generalized derivations acting on a prime ideal P and some applications},
	Indian J. Pure Appl. Math. {\bf 53}(3) (2022), 792-800.
	
	\bibitem{M.Ashraf2002}
	{\sc M. Ashraf, N. Rehman}, {\sl On commutativity of prime rings with derivations}, Result. Math., \textbf{42} (2002), 3-8.
	
	

	
	\bibitem{bresar1989}
	{\sc M. Bre\v{s}ar, J. Vukman},
	{\sl On some additive mappings in rings with involution},
	Aequationes Mathematicae, {\bf 38} (1989), 178-185.
	
 \bibitem{Khan2021}
M. S. Khan, S. Ali and M. Ayedh,
{\it Herstein's theorem for prime ideals in rings with involution involving pair of derivations.}
Commun. Algebra {\bf 50}(6) (2022), 2592-2603.
	
	\bibitem{Samman2006} {\sc M. Samman}, {\sl A note on reverse derivations}, International Journal of Mathematical Education in Science and Technology, \textbf{37} (2006), 98-101.
	
	
	
	\bibitem{Rehman24}
	{\sc N. U. Rehman, H. M. Alnoghashi, and M. Hongan},
	{\sl On generalized derivations involving prime ideals with involution.}
	Ukrainian Mathematical Journal. {\bf 75}(08) (2024), 1219-1241.
	
	
	

	
	
	
	
\end{thebibliography}
\end{document}